\documentclass[11pt,a4paper]{article}
\author{Chantal Klinkhammer$^{\ast}$ \hspace{1cm} Robert Plato\footnote{Department of Mathematics, University of Siegen, 57068 Siegen, Germany}}
\newcommand{\lVert}{\left\|}
\newcommand{\rVert}{\right\|}
\usepackage{amsthm} 
\usepackage{amsmath}
\usepackage{amssymb} 
\usepackage{nicefrac}
\usepackage{booktabs}
\usepackage{microtype} 
\usepackage{float}
\usepackage[ruled,vlined,]{algorithm2e}
\usepackage{cite}
\usepackage{xcolor}
\usepackage{hyperref}
\usepackage{enumitem}
\newtheorem{theorem}{Theorem}[section]
\newtheorem{lemma}[theorem]{Lemma}
\newtheorem{cor}[theorem]{Corollary}

  \theoremstyle{definition}
\newtheorem{remark}[theorem]{Remark}
\newtheorem{defn}[theorem]{Definition}
\newtheorem{assumption}[theorem]{Assumption}

\newcommand{\rme}{\mathrm{e}}
\newcommand{\rmd}{\mathrm{d}}
\newcommand{\Or}{\mathrm{O}}

\begin{document}
\date{}
\title{Analysis of the discrepancy principle for
Tikhonov regularization under low order source conditions}
\maketitle


\begin{abstract}
We study the application of Tikhonov regularization to ill-posed nonlinear operator equations. The objective of this work is to prove low order convergence rates for the discrepancy principle under low order source conditions of logarithmic type. 
We work within the framework of Hilbert scales and extend existing studies on this subject to the oversmoothing case. 
The latter means that the exact solution of the treated operator equation does not belong to the domain of definition of the penalty term. As a consequence, the Tikhonov functional fails to have a finite value.
\end{abstract}

%


\section{Introduction}

Suppose that the solution of an ill-posed operator equation is not as smooth as a solution obtained by regularization. Does the regularization method still yield order optimal convergence rates? 
We answer this question for a particular setting that is specified below. 

In what follows, $X$ and $Y$ are infinite-dimensional Hilbert spaces with norms $\| \cdot \|$. 
The norms of the spaces may be different; however, if the pertaining space can be identified from the context, we omit the indices subsequently. 
We consider the operator equation
\begin{equation}
Fu=f^{\dagger}\label{E:OperatorEquation}
\end{equation}
for a nonlinear operator $F:\mathcal{D}(F)\rightarrow Y,$ with domain $\mathcal{D}(F)\subset X$.
Thereby, the right-hand side $f^{\dagger}$ is approximately given by noisy observations $f^{\delta} \in Y$ satisfying
\[
\| f^{\delta}-f^{\dagger}\| \leq \delta
\]
for some deterministic noise level $\delta\geq 0$.
We assume that a solution $u^{\dagger}\in \mathcal{D}(F)$ of the
operator equation (\ref{E:OperatorEquation}) exists. In addition, the equation is supposed to be at least locally ill-posed at $u^{\dagger}$.
The concept of local ill-posedness is understood as in \cite{HofmannPlato2018}.

An appropriate way to find approximate solutions for $u^{\dagger}$ is by applying the following version of Tikhonov regularization.
For a regularization parameter $\alpha>0$ and a given initial guess $\overline{u}$, we minimize the functional
\[
T_{\alpha}^{\delta}(u):=\left\| Fu-f^{\delta}\right\|^{2}+\alpha\left\| u-\overline{u}\right\|_{1}^{2},
\]
subject to $u\in \mathcal{D}(F)$.
The minimizers, denoted by $u_{\alpha}^{\delta}$, provide regularized solutions to $u^{\dagger}$.
The norm $\lVert \cdot \rVert_{1}$, occurring in the penalty term of $T_{\alpha}^{\delta}$, is assumed to be stronger than the norm in $X$, and it is the norm equipping a densely defined subspace $X_{1}\subset X$. We focus on the instance of an oversmoothing penalty term, which is evoked by $u^{\dagger}\notin X_{1}$ such that $T_{\alpha}^{\delta}(u^{\dagger})=\infty$.

In this setting, studies have established convergence rates for Tikhonov regularization with H\"older type source conditions; see \cite{HofmannMatheApriori2018} for an a priori parameter choice and \cite{HofmannMathe2017} for an a posteriori parameter choice of $\alpha$. The present study examines the case of low order source conditions. As a less restrictive smoothness assumption, this type of source condition has been gaining interest since its treatment in \cite{Hohage1999}. The work in \cite{HofmannPlato2020} presents convergence rates for low order source conditions and an a priori parameter choice. We extend these studies by providing results for an a posteriori parameter choice of $\alpha$. More precisely, we use the discrepancy principle. 

We begin the next section by presenting the concept of Hilbert scales, which allows us to formulate a smoothing property of the operator $ F $.
The section also presents the main assumptions, including the smoothing property, and specifies the source condition.
In the third section, we introduce auxiliary elements. These are useful since the standard estimate $T_{\alpha}^{\delta}(u_{\alpha}^{\delta})\leq T_{\alpha}^{\delta}(u^{\dagger})$
cannot be utilized in proofs if oversmoothing penalty terms are involved.
In Section \ref{S:DiscrepancyPrinciple}, we examine the discrepancy principle,
and in Section \ref{S:MainResult}, we present a proof of our main result.
In Section \ref{S:NumericalResults}, we finish this paper with numerical results, which
confirm the statement of the main theorem.
\section{Prerequisites}\label{S:Prerequisites}
This section consists of three parts. In the first part, we introduce Hilbert scales. Afterwards, in Subsection \ref{SS:Asssumptions} and Subsection \ref{SS:Source condition}, 
{the main assumptions and the source condition are stated, respectively. The latter can be considered as an assumption regarding the solution smoothness.
\subsection{Hilbert scales}
In our framework, we can simplify the definition of Hilbert scales as follows:
Let
\[
B:\mathcal{D}(B)\to X
\] 
be a selfadjoint, unbounded linear operator with dense domain $\mathcal{D}(B)\subset X$.
The condition
\begin{equation}\label{E:boundedness of B inverse}
\lVert Bu\rVert \geq k\lVert u\rVert,
\end{equation}
for $u\in \mathcal{D}(B)$
{,} 
and a positive constant $k$, ensures that $B$ is strictly positive and has a bounded inverse.
We set
\begin{equation*}
X_{\tau}:=
X=
\begin{cases}
\mathcal{D}(B^{\tau})& \quad \text{for } \tau >0\\
X& \quad \text{for } \tau \leq 0\\
\end{cases}
\end{equation*}
with norms $\lVert u \rVert_{\tau}=\lVert B^{\tau}u \rVert$ for $u \in X_{\tau}$ and call the system of spaces $\left(X_{\tau}\right)_{\tau\in \mathbb{R}}$ Hilbert scale.
For a definition of the fractional powers and a detailed introduction to Hilbert scales, see \cite{EnglHankeNeubauer2000}.

The work within Hilbert scales provides a versatile tool
called interpolation inequality, which is stated next. The result and its proof can be found in \cite{Baumeister1987}.
\begin{lemma}[Interpolation inequality]\label{L:InterpolationInequal}
For all $p\geq q\geq s$, $p\neq s$, and $x \in X_{p}$, we have
\begin{equation}\label{E:InterpolationInequality}
\lVert x \rVert_{q} \leq \lVert x \rVert_{p}^{1-\psi}\lVert x\rVert_{s}^{\psi},
\end{equation}
where $\psi=\frac{p-q}{p-s}$.
\end{lemma}

\subsection{Assumptions}\label{SS:Asssumptions}
In this subsection, we formulate the main assumptions.
These are adopted from \cite{HofmannPlato2020}.
\begin{assumption}\label{Assumption}
\begin{itemize}
\item[$(a)$] The operator $F:\mathcal{D}(F)\to Y$ is sequentially continuous on $\mathcal{D}(F)$ with respect to the weak topologies of the Hilbert spaces $X$ and $Y$.
\item[$(b)$] The domain $\mathcal{D}(F)$ is a closed and convex subset of $X$.
\item[$(c)$] $\mathcal{D}:=\mathcal{D}(F) \cap X_{1}\neq \emptyset$  and $\overline{u}\in \mathcal{D}$.
\item[$(d)$] The solution $u^{\dagger}\in \mathcal{D}(F)$ is an interior point of $\mathcal{D}(F)$.
\item[$(e)$] The observations $f^{\delta}$ satisfy $\lVert f^{\delta}-f^{\dagger}\rVert \leq \delta$.
\item[$(f)$] For $a>0$ and constants $c_{a},C_{a}\in (0,\infty)$ with $c_{a}\leq C_{a}$, the condition 
\begin{equation}\label{E:nonlinearityCondition}
c_{a}\lVert u-u^{\dagger}\rVert_{-a}\leq \lVert Fu-f^{\dagger}\rVert \leq C_{a}\lVert u-u^{\dagger}\rVert_{-a}
\end{equation}

is satisfied for all $u\in \mathcal{D}$.
\end{itemize}
\end{assumption}
We interpret the inequality chain in (\ref{E:nonlinearityCondition}) as a smoothing property. However, sometimes it is referred to as nonlinearity condition; see e.g. \cite{HofmannMathe2017,HofmannPlato2020}.
Based on the assumptions in \ref{Assumption}, we can derive some properties of the regularized solutions $u_{\alpha}^{\delta}$.
For verification, we refer to Proposition 2.4 and Remark 2.5 in \cite{HofmannPlato2020}.
\begin{remark}
\begin{itemize}
\item With the assumptions in \ref{Assumption}, we ensure that for each $\alpha>0$, a minimizer $u_{\alpha}^{\delta}\in \mathcal{D}$ of $T_{\alpha}^{\delta}$ exists.
\item The minimizer may be not unique because the 
misfit functional $\lVert Fu-f^{\delta} \rVert^2$ 
and hence the Tikhonov functional $T_{\alpha}^{\delta}$ may be non-convex.
\item Each minimizing sequence of $T_{\alpha}^{\delta}$ over $\mathcal{D}$ has a subsequence that converges strongly in $X_{1}$ to a minimizer $u_{\alpha}^{\delta}\in\mathcal{D}$ of the Tikhonov functional.
\item For every $\alpha>0$, the regularized solutions $u_{\alpha}^{\delta}$ are stable in $X_{1}$ with respect to small
perturbations in the data $f^{\delta}$.

\end{itemize}
\end{remark}

\subsection{Source condition}\label{SS:Source condition}
We complete this section by composing a smoothness assumption on the solution $u^{\dagger}$ by virtue of a source condition. 
For this, we define the linear operator $G$ by 
\[
G:X\to X, \qquad 
G:=B^{-(2a+2)},
\]
where $a>0$ is determined by the smoothing property (\ref{E:nonlinearityCondition}). The operator $G$ is bounded, injective, selfadjoint, and positive semidefinite.
From now on, we assume that the solution $u^{\dagger}$ obeys the source condition
\begin{equation}\label{E:SourceCondition}
u^{\dagger}-\overline{u}=\varphi(G)w,
\end{equation}
for some source element $w\in X$ with $\lVert w\rVert \leq \rho$, $\rho>0$.
Here, the function $\varphi$ is defined as
\begin{equation}\label{E:Phi}
\varphi: (0,\lVert G\rVert] \to (0,\infty), \qquad 
t\mapsto(-\ln (c t))^{-\kappa},
\end{equation}
for $\kappa>0$ and with 
\begin{equation}\label{E:c}
c\in (0,\lVert G \rVert^{-1}),
\end{equation}
such that the source condition is of logarithmic type and hence of low order. 
The requirement in (\ref{E:c}) ensures that $\varphi$ does not attain singularities. Throughout this paper, the constant $c$ addresses (\ref{E:Phi}).

\section{Auxiliary elements}
In this section, we define auxiliary elements and derive some related results,
which are crucial for the proof of the main result. 
As auxiliary elements, also referred to as smooth approximations (see e.g.~\cite{HohageMiller2022}), we consider the minimizers $\hat{u}_{\alpha}$ of the artificial Tikhonov functional
\begin{equation*}
T_{a,\alpha}(u):=\Vert u- u^{\dagger}\Vert_{-a}^{2}+\alpha\lVert u-\overline{u} \rVert_{1}^{2}.
\end{equation*}
These are uniquely defined, and it can be shown that they admit the representation
\begin{equation}\label{E:Representation}
\hat{u}_{\alpha}:=\overline{u}+G(G+\alpha I )^{-1}(u^{\dagger}-\overline{u})=u^{\dagger}-\alpha(G+\alpha I)^{-1}(u^{\dagger}-\overline{u}) ,
\end{equation}
for $\alpha>0$.
This representation and the following Corollary~\ref{C:CorIndexAbsch} 
allow us to make the statements in Lemma~\ref{L:ErsteAbsch}.
Thereby, Corollary~\ref{C:CorIndexAbsch} is an adapted formulation of Corollary~5.2 in \cite{HofmannPlato2020}. The proof therein remains valid, so we refer the interested reader to the aforementioned paper. 
\begin{cor}\label{C:CorIndexAbsch}
Let $f:(0,\lVert G\rVert]\to (0,\infty)$ be a continuous, monotonically non-decreasing function with
\[
\lim_{t\downarrow 0}f(t)=0.
\] 
Moreover, suppose that for each exponent $\eta$ and sufficiently small $t>0$, the quotient function $t\mapsto t^{\eta}/f(t)$ is strictly increasing.
Then, for each $\theta \in [0,1)$, there exist constants $\overline{\alpha}>0$ and $\overline{C}>0$ such that 
\[
\underset{\lambda \in \left(0, \lVert G \rVert\right]}{\sup}\frac{\alpha \lambda^{\theta}f(\lambda)}{\lambda+\alpha}\leq \overline{C}\alpha^{\theta}f(\alpha),
 \qquad
\text{for }  \alpha \in (0,\overline{\alpha}].\]
\end{cor}
\begin{remark}
The function $\varphi$, defined as in (\ref{E:Phi}), fulfils the conditions of Corollary \ref{C:CorIndexAbsch}:

Obviously, $\varphi$ is a continuous, monotonically non-decreasing function and satisfies $
\lim_{t\downarrow 0}\varphi(t)=0$.
  Considering $t\in(0,c^{-1}\rme^{-\kappa/\eta})$, the last required property can be observed as a result of
\begin{align*}
\frac{\rmd}{\rmd t}\left(\frac{t^{\eta}}{(-\ln (ct))^{-\kappa}}\right)
&=
\eta t^{\eta -1}(-\ln (ct))^{\kappa}-t^{\eta}\kappa(-\ln(c t))^{\kappa -1}\frac{1}{t}\\
&=
t^{\eta-1}(-\ln (ct))^{\kappa -1}\left[\eta (-\ln(c t))-\kappa\right]
> 0.
\end{align*}
Thus, for
$t< \Vert G \Vert \rme^{-\kappa/\eta}$,
the quotient function $t\mapsto t ^{\eta}/\varphi(t)$ is strictly increasing. 
\end{remark}
Note that in the following, 
we use the abbreviation 
\[
r:=\frac{2a+2}{a}.
\]
The next result in Lemma \ref{L:ErsteAbsch} is used to establish the subsequent lemmas. It provides bounds for terms involving the auxiliary elements.  
\begin{lemma}\label{L:ErsteAbsch}
There exist positive constants $C_{i}$, $i=1,2,3$, and $\alpha_{0}\leq \lVert G\rVert$ such that the following inequalities hold for $\alpha \in (0,\alpha_{0}]$: 
\begin{enumerate}[label=(\roman*)]
\item \label{En:itemiLemma33}
$
\lVert \hat{u}_{\alpha}-u^{\dagger} \rVert \leq C_{1}\varphi(\alpha),
$
\item \label{En:itemiiLemma33}
$
\lVert \hat{u}_{\alpha}-u^{\dagger} \rVert_{-a} \leq C_{2}\alpha^{\frac{1}{r}}\varphi(\alpha),
$
\item \label{En:itemiiiLemma33}
$\lVert \hat{u}_{\alpha}-\overline{u} \rVert_{1} \leq  C_{3}\alpha^{-\frac{1}{ra}}\varphi(\alpha).
$
\end{enumerate}
\end{lemma}

\begin{proof}
To deduce the items of the lemma, we use representation (\ref{E:Representation}), source condition~(\ref{E:SourceCondition}), and Corollary \ref{C:CorIndexAbsch}.
We first examine item $\ref{En:itemiLemma33}$. According to Corollary~\ref{C:CorIndexAbsch}, there exist positive constants $\overline{C}_{1}$ and $\overline{\alpha}_{1}$ such that
\begin{align*}
\lVert \hat{u}_{\alpha}-u^{\dagger} \rVert
&=
\lVert -\alpha(G+\alpha I)^{-1}(u^{\dagger}-\overline{u}) \rVert
=
\lVert \alpha(G+\alpha I)^{-1}\varphi(G)w \rVert
\\
& \leq
\overline{C}_{1}\rho\varphi(\alpha)=C_{1}\varphi(\alpha), \qquad \alpha \in (0,\overline{\alpha}_{1}].
\end{align*}
In the same manner, we obtain for item $\ref{En:itemiiLemma33}$
\begin{align*}
\lVert \hat{u}_{\alpha}-u^{\dagger} \rVert_{-a}
&=
\lVert B^{-a}\left(\hat{u}_{\alpha}-u^{\dagger}\right)\rVert
=
\lVert G^{\frac{a}{2a+2}}\left(\hat{u}_{\alpha}-u^{\dagger}\right)\rVert \\
&=
\Vert \alpha G^{\frac{1}{r}}(G+\alpha I)^{-1}\varphi(G)w \Vert
\leq C_{2}\alpha^{\frac{1}{r}}\varphi(\alpha), \qquad \alpha \in (0,\overline{\alpha}_{2}]
\end{align*}
and for item $\ref{En:itemiiiLemma33}$
\begin{align*}
\lVert \hat{u}_{\alpha}-\overline{u} \rVert_{1} 
&= \lVert B\left(\hat{u}_{\alpha}-\overline{u}\right) \rVert
=\lVert \frac{1}{\alpha}\alpha G^{-\frac{1}{2a+2}}G(G+\alpha I )^{-1}(u^{\dagger}-\overline{u}) \rVert \\
&\leq C_{3}\alpha^{-\frac{1}{ra}}\varphi(\alpha), \qquad \alpha \in (0,\overline{\alpha}_{3}],
\end{align*}
with $\overline{\alpha}_{2},\overline{\alpha}_{3}>0$ and $C_{2}=\overline{C}_{2}\rho,C_{3}=\overline{C}_{3}\rho>0$, chosen accordingly. 
Setting $\alpha_{0}=\min \left\{ \overline{\alpha}_{1},\overline{\alpha}_{2},\overline{\alpha}_{3}\right\}$ yields the assertion.
\end{proof}
With Lemma~\ref{L:ErsteAbsch}, we are ready to establish the next result in Lemma~\ref{L:ZweiteAbsch}, which is used at several points in this paper.
\begin{lemma}\label{L:ZweiteAbsch}
Under Assumption \ref{Assumption}, the following inequality is satisfied:
\[
\max\left\{ \lVert Fu_{\alpha}^{\delta}-f^{\delta}\rVert, \sqrt{\alpha}\lVert u_{\alpha}^{\delta}-\overline{u}\rVert_{1} \right\} \leq C_{4} \alpha^{\frac{1}{r}}\varphi(\alpha)+\delta,
\]
for a positive constant $C_{4}$ and $\alpha\leq \alpha_{4}$, with $\alpha_{4}$ sufficiently small.
\end{lemma}
For a proof, we follow along the lines of the proof of Lemma 3.2 in \cite{HofmannPlato2020}.
\begin{proof}
Let $\alpha_{0}$ be the constant from Lemma \ref{L:ErsteAbsch}. From item $\ref{En:itemiLemma33}$ of this lemma, it follows that $\hat{u}_{\alpha}\in \mathcal{D}$ for $\alpha>0$ sufficiently small, say $\alpha\leq \alpha_{4}$ for some $\alpha_{4}\leq \alpha_{0}$,
because $u^{\dagger}$ is assumed to be an interior point of $\mathcal{D}(F)$. Based on this information, we can compute
\begin{align*}
 \left(\lVert F u_{\alpha}^{\delta}-f^{\delta}\rVert^{2}+\alpha\lVert u_{\alpha}^{\delta}-\overline{u}\rVert_{1}^{2}\right)^{\frac{1}{2}}
&\leq 
\left(\lVert F \hat{u}_{\alpha}-f^{\delta}\rVert^{2}+\alpha\lVert \hat{u}_{\alpha}-\overline{u}\rVert_{1}^{2}\right)^{\frac{1}{2}} \\
 &\leq
\lVert F\hat{u}_{\alpha}-f^{\delta}\rVert+\sqrt{\alpha}\lVert \hat{u}_{\alpha}-\overline{u}\rVert_{1} \\
&\leq
\lVert F\hat{u}_{\alpha}-f^{\dagger}\rVert+\sqrt{\alpha}\lVert \hat{u}_{\alpha}-\overline{u}\rVert_{1}+\delta, \ \alpha\leq\alpha_{4}.
\end{align*}
The right-hand side of (\ref{E:nonlinearityCondition}) and item $\ref{En:itemiiLemma33}$ of Lemma \ref{L:ErsteAbsch} allow us to estimate the first summand:
\[
\lVert F\hat{u}_{\alpha}-f^{\dagger}\rVert \leq C_{a}\lVert \hat{u}_{\alpha}-u^{\dagger}\rVert_{-a} \leq C_{a}C_{2}\alpha^{\frac{1}{r}}\varphi(\alpha), \qquad \alpha\leq\alpha_{4}.
\]
Using item $\ref{En:itemiiiLemma33}$ of Lemma \ref{L:ErsteAbsch}, we find for the second summand 
\[
\sqrt{\alpha}\lVert \hat{u}_{\alpha}-\overline{u}\rVert_{1}\leq C_{3}\sqrt{\alpha}\alpha^{-\frac{1}{ra}}\varphi(\alpha)=C_{3}\alpha^{\frac{1}{2}-\frac{1}{2a+2}}\varphi(\alpha), \qquad \alpha\leq\alpha_{4}.
\]
With $C_{4}=C_{a}C_{2}+C_{3}$ the claim holds.
\end{proof}
In Lemma \ref{L:ABC}, we introduce a parameter choice that allows us to determine upper bounds for the items in Lemma \ref{L:ErsteAbsch} by functions of $\delta$. 
\begin{lemma}\label{L:ABC}
For the a priori parameter choice 
\[
\beta=\beta(\delta):=\delta^{r}\left(-\ln(c\delta)\right)^{\kappa r},
\]
we consider the auxiliary elements $\hat{u}_{\beta}$, defined as in (\ref{E:Representation}), with $\alpha=\beta$. 
Then, the following bounds hold as $\delta \downarrow 0$:
\begin{enumerate}[label=(\roman*)]
\item \label{En:itemiLemmaABC}
$
\lVert \hat{u}_{\beta}-u^{\dagger} \rVert= \Or\left(\varphi(\delta)\right),
$
\item \label{En:itemiiLemmaABC}
$
\lVert \hat{u}_{\beta}-u^{\dagger} \rVert_{-a}  = \Or(\delta), 
$
\item \label{En:itemiiiLemmaABC}
$
\lVert \hat{u}_{\beta}-\overline{u} \rVert_{1} =\Or\left(\delta^{-\frac{1}{a}}\varphi(\delta)^{\frac{r}{2}}\right).
$
\end{enumerate}
\end{lemma}

\begin{proof}
According to Lemma \ref{L:ErsteAbsch}, we have
\begin{equation*}
 \lVert \hat{u}_{\beta}-u^{\dagger} \rVert \leq C_{1}\varphi(\beta), 
\
\lVert \hat{u}_{\beta}-u^{\dagger} \rVert_{-a} \leq C_{2}\beta^{\frac{1}{r}}\varphi(\beta), 
\text{ and }
\lVert \hat{u}_{\beta}-\overline{u} \rVert_{1} \leq  C_{3}\beta^{-\frac{1}{ra}}\varphi(\beta),
\end{equation*}
for some constants $C_{i}$, $i=1,2,3$, and with $\beta$ small enough. 
Consequently, it is sufficient to show the estimates
\begin{enumerate}[label=$(\Alph*)$]
\item \label{EN:itemAProof}
$
\varphi(\beta)= \Or\left(\varphi(\delta)\right),
$
\item \label{EN:itemBProof}
$
\beta^{\frac{1}{r}}\varphi(\beta) = \Or(\delta), \text{ and }
$
\item \label{EN:itemCProof}
$
\beta^{-\frac{1}{ra}}\varphi(\beta) =\Or\left(\delta^{-\frac{1}{a}}\varphi(\delta)^{\frac{r}{2}}\right)
$
\end{enumerate}
as $\delta\downarrow 0$.
First, we look at item \ref{EN:itemAProof}:
\begin{align}
\dfrac{\varphi(\beta)}{\varphi(\delta)}
&=
\dfrac{\left[-\ln\left(c\delta^{r}\left(-\ln(c\delta)\right)^{\kappa r}\right)\right]^{-\kappa}
}{(-\ln(c\delta))^{-\kappa}} \nonumber\\
&=
\left[\dfrac{r\ln(c\delta)+\ln(c^{1-r})+\kappa r (\ln(-\ln(c\delta)))}{\ln(c\delta)}\right]^{-\kappa} \nonumber \\ 
 &=
\left[ r \left(1+ \dfrac{r^{-1}\ln(c^{1-r})+\kappa \ln(-\ln(c\delta))}{\ln(c\delta)}\right)\right]^{-\kappa}. \label{E:Klammer}
\end{align}
Consequently, 
\begin{equation*}
\lim_{\delta\downarrow 0} \frac{\varphi(\beta) }{ \varphi(\delta)} =r^{-\kappa}.
\end{equation*}
Next, consider item \ref{EN:itemBProof}:
\begin{align*}
\frac{\beta^{\frac{1}{r}}\varphi(\beta)}{\delta} 
&=
\dfrac{\delta
(-\ln(c\delta))^{\kappa}\left[ -\ln \left( c \delta^{r}(-\ln(c\delta))^{\kappa r}\right)\right]^{-\kappa}}{\delta}\\
&=
\left[ r \left( 1+\dfrac{r^{-1}\ln(c^{1-r})+\kappa \ln(-\ln(c\delta))}{\ln(c\delta)}\right)\right]^{-\kappa}.
\end{align*}
This term coincides with (\ref{E:Klammer}). Hence we have $\lim_{\delta\downarrow0}\frac{\beta^{\frac{1}{r}}\varphi(\beta)}{\delta}=r^{-\kappa}$.
Finally, we examine item \ref{EN:itemCProof}:
\begin{align*}
 \dfrac{\beta^{-\frac{1}{ra}}\varphi(\beta)}{\delta^{-\frac{1}{a}}\varphi(\delta)^{\frac{r}{2}}}  
&=
\dfrac{\delta^{-\frac{1}{a}}(-\ln(c\delta))^{-\frac{\kappa}{a}}\left[-\ln\left(c\delta^{r}(-\ln(c\delta))^{\kappa r}\right)\right]^{-\kappa}}{\delta^{-\frac{1}{a}}\left(-\ln(c\delta)\right)^{-\kappa\frac{r}{2}}} \\
 &=
\left[\dfrac{-\ln\left(c\delta^{r}(-\ln(c\delta))^{\kappa r}\right)}{\left(-\ln(c\delta)\right)^{-\frac{1}{a}+\frac{r}{2}}}\right]^{-\kappa}
\\
 &=
\left[ r \left( 1+\dfrac{r^{-1}\ln(c^{1-r})+\kappa \ln(-\ln(c\delta))}{\ln(c\delta)}\right)\right]^{-\kappa}.
\end{align*}
Again, the same term as in (\ref{E:Klammer}) results such that the whole expression converges to $r^{-\kappa}$ as $\delta \downarrow 0$.
\end{proof}

\section{Discrepancy principle and associated results}\label{S:DiscrepancyPrinciple}
In the present context, the following version of the discrepancy principle is used to select the regularization parameter $\alpha$.
\begin{defn}(Discrepancy principle)\label{D:Discrepancy}
For constants $k,l \in (1,\infty)$ proceed as follows:
\begin{itemize}
\item[$(a)$] If $\lVert F\overline{u}-f^{\delta}\rVert \leq k \delta$ holds, choose $u_{\alpha_{\ast}}^{\delta}:=\overline{u}\in \mathcal{D}$, which corresponds to $\alpha_{\ast}=\infty$.

\item[$(b)$] If $\lVert F\overline{u}-f^{\delta}\rVert > k \delta$, determine $\alpha=:\alpha_{\ast}\in(0,\infty)$ such that
\begin{equation} \label{E:DP}
\lVert Fu_{\alpha_{\ast}}^{\delta}-f^{\delta}\rVert \leq k\delta \leq \lVert Fu_{\gamma_{\ast}}^{\delta}-f^{\delta}\rVert,
\end{equation}
for some $ \gamma_{\ast}\in [\alpha_{\ast}, l\alpha_{\ast}]$.
\end{itemize}
\end{defn}
\begin{remark}
\begin{enumerate}
\item The choice $u_{\infty}^{\delta}=\overline{u}$ in item $(a)$ of Definition \ref{D:Discrepancy} is coherent because we have $\lim_{\alpha\rightarrow\infty}\lVert u_{\alpha}^{\delta}-\overline u\rVert=0$. This can be seen through the following computation:
\[
\lVert Fu_{\alpha}^{\delta}-f^{\delta}\rVert^{2}+\alpha\lVert u_{\alpha}^{\delta}-\overline u\rVert_{1}^{2}
=T_{\alpha}^{\delta}(u_{\alpha}^{\delta})
\leq
T_{\alpha}^{\delta}(\overline{u})
=
\lVert F \overline{u}-f^{\delta}\rVert^{2}.
\]
It follows $\lVert u_{\alpha}^{\delta}-\overline u\rVert_{1}=\Or
(\alpha^{-\frac{1}{2}})$ 
as $\alpha \rightarrow \infty$ and consequently, due to condition (\ref{E:boundedness of B inverse}) of the operator $B$,
that $\lVert u_{\alpha}^{\delta}-\overline u\rVert=\Or
(\alpha^{-\frac{1}{2}})$ 
as $\alpha\rightarrow\infty$.
\item The discrepancy principle in Definition \ref{D:Discrepancy} allows discontinuities in the functional $\lVert Fu_{\alpha}^{\delta}-f^{\delta}\rVert$ and is thus more flexible than the classical version of the discrepancy principle, i.e.~determining a parameter $\alpha$ such that $\lVert Fu_{\alpha}^{\delta}-f^{\delta}\rVert\approx k\delta$.
Moreover, Definition \ref{D:Discrepancy} allows a sequential implementation as we see in Section \ref{S:NumericalResults}.
\item Under Assumption \ref{Assumption}, the function $\alpha\mapsto\lVert Fu_{\alpha}^{\delta}-f^{\delta}\rVert$ is monotonically non-decreasing for fixed values of $\delta>0$. 
For a proof, see \cite{HofmannPlato2020}. This fact assures that a parameter $\alpha_{\ast}$ can be found such that it determines an approximation $u_{\alpha_{\ast}}^{\delta}\in\mathcal{D}$ for each fixed noise level $\delta$. Thus, the discrepancy principle is appropriate. 
\end{enumerate}
\end{remark}
The next result enables us to derive the lower bound in Lemma \ref{L:LowerEstAlphaAst} for $\alpha_{\ast}$ chosen according to the parameter choice strategy specified in Definition \ref{D:Discrepancy}.
\begin{lemma}\label{L:Tautenhahn}
For constants $b,d >0$, the inverse of the function 
\begin{equation}\label{E:chi}
\chi_{b,d}:(0,\lVert G\rVert]\to \mathbb{R}, \qquad t\mapsto t^{\frac{1}{b}}\left(-\ln (ct)\right)^{-d}
\end{equation} 
satisfies
\begin{equation*}
\chi^{-1}_{b,d}(t)\geq C_{5}t^{b}\left(-\ln (ct)\right)^{bd}, \qquad t \in (0,\lVert G\rVert]
\end{equation*}
for some constant $C_{5}>0$.
\end{lemma}
\begin{proof}
The basic idea of this proof stems from the proof of Lemma~3.3 in \cite{Tautenhahn1998}.
Noticing that $\chi_{b,d}$ is continuous and strictly monotonically increasing with $\lim_{t\downarrow 0}\chi_{b,d}(t)=0$, we define $\chi^{-1}_{b,d}(t)=\lambda$, which can be written as
\begin{eqnarray*}
\chi^{-1}_{b,d}(t)=t^{b}(-\ln(c t))^{bd}\frac{\lambda}{t^{b}(-\ln(c t))^{bd}} .
\end{eqnarray*}
Substituting $t=\chi_{b,d}(\lambda)$ in the fraction, we derive
\begin{align*}
 \frac{\lambda}{t^{b}(-\ln (ct))^{bd}}
&=
\dfrac{\lambda}{\left(\lambda^{\frac{1}{b}}\left(-\ln (c\lambda)\right)^{-d}\right)^{b}\left(-\ln \left(c\lambda^{\frac{1}{b}}\left(-\ln (c\lambda)\right)^{-d}\right)\right)^{bd}} \\
&=
 \left[\dfrac{-\ln \left(c \lambda^{\frac{1}{b}}\left(-\ln (c\lambda)\right)^{-d}\right)}{-\ln (c\lambda)}\right]^{-bd} \\
&=
\left[\dfrac{\frac{1}{b}\ln (c\lambda)+\ln(c^{1-\frac{1}{b}}) -d \ln \left(-\ln (c\lambda)\right)}{\ln(c \lambda)}\right]^{-bd}\\
&=
\left[\dfrac{1}{b}+\dfrac{ \ln(c^{1-\frac{1}{b}})-d \ln \left(-\ln(c \lambda)\right)}{\ln(c \lambda)}\right]^{-bd}.
\end{align*}
Now as $\lambda\downarrow 0$, or equivalently $t\downarrow 0$,
we have 
\[
\lim_{\lambda\downarrow 0}\left[\frac{1}{b}+\frac{ \ln(c^{1-\frac{1}{b}})-d \ln \left(-\ln(c \lambda)\right)}{\ln(c \lambda)}\right]^{-bd}=b^{bd}.
\]
It follows 
\[
\chi^{-1}_{b,d}(t)
=t^{b}(-\ln (ct))^{bd}(b^{bd}+\mathrm{o}(1)) \qquad \text{as }t\downarrow 0.
\]
Arguing that $\chi^{-1}_{b,d}$ is continuous on the compact interval $[\varepsilon,\Vert G\Vert]$, for $ \varepsilon > 0 $ small, yields the existence of a constant $C_{5}>0$ such that
\[ 	
\chi^{-1}_{b,d}(t)
\geq C_{5}t^{b}(-\ln(c t))^{bd}, \qquad t\in (0,\Vert G\Vert].
\]
\end{proof}
\begin{lemma}[Lower bound for $\alpha_{\ast}$]\label{L:LowerEstAlphaAst}
Under Assumption \ref{Assumption}, there exist constants $C,\delta_{0}>0$ such that the regularization parameter $\alpha_{\ast}$ chosen according to Definition~\ref{D:Discrepancy} satisfies
\begin{eqnarray*}
\alpha_{\ast}\geq C \delta^{r} \left(- \ln (c\delta)\right)^{\kappa r} \qquad \text{for } \delta \leq \delta_{0},
\end{eqnarray*}
if item (b) of that definition applies.
\end{lemma}

\begin{proof}
Let $ C_{4}$ and $\alpha_{4}$ be the constants from Lemma~\ref{L:ZweiteAbsch}. Further, let $k,l$, and $C_{5}$ be the constants from Definition~\ref{D:Discrepancy} and Lemma~\ref{L:Tautenhahn}, respectively. 
We consider the case $l\alpha_{\ast}\in(0,\alpha_{4}]$ first. With item $(b)$ of Definition \ref{D:Discrepancy} and Lemma~\ref{L:ZweiteAbsch}, we conclude that
\begin{align*}
k\delta \leq \lVert Fu_{\gamma_{\ast}}^{\delta}-f^{\delta}\rVert  \leq C_{4}\gamma_{\ast}^{\frac{1}{r}}\varphi(\gamma_{\ast})+\delta
&\leq C_{4}(l\alpha_{\ast})^{\frac{1}{r}}\varphi(l\alpha_{\ast})+\delta \\
&=C_{4}\chi_{r,\kappa}(l\alpha_{\ast})+\delta,
\end{align*}
where 
$
\chi_{r,\kappa},
$
defined as in (\ref{E:chi}), is monotonically non-decreasing owing to the constraint $l\alpha_{\ast}\leq\alpha_{4}\leq \lVert G\rVert $. Because $k>1$, we can write
\[
C_{4}^{-1}(k-1)\delta \leq \chi_{r,\kappa}(l\alpha_{\ast})
\]
and use the estimate in Lemma~\ref{L:Tautenhahn} to obtain
\[
l\alpha_{\ast}\geq C_{5}(C_{4}^{-1}(k-1)\delta)^{r}\left(- \ln (cC_{4}^{-1}(k-1)\delta)\right)^{\kappa r}, \qquad C_{4}^{-1}(k-1)\delta \leq\Vert G\Vert.
\]
Note that
\[
\ln(c\delta)\sim \ln (cC_{4}^{-1}(k-1)\delta) \qquad \text{as }\delta \downarrow 0,
\]
where the symbol $\sim$ indicates asymptotic equivalence in the sense of
\[
\lim_{\delta\downarrow 0}\frac{\ln(c\delta)}{ \ln (cC_{4}^{-1}(k-1)\delta)}=1.
\]
For this reason, we have
\[
\alpha_{\ast}\geq \frac{C_{5}}{l}\left(\frac{k-1}{C_{4}}\right)^{r} \delta^{r}\left(- \ln (c\delta)\right)^{\kappa r} 
\]
with $\delta$ small enough, say $\delta\leq t_{2}$ for some $t_{2} \in (0, \frac{C_{4}\Vert G\Vert}{k-1}]$.

In the case $l\alpha_{\ast}\in( \alpha_{4},\infty)$
the assertion still follows, because 
\begin{eqnarray*}
&l\alpha_{\ast}&> \alpha_{4}
\\
\Leftrightarrow
&\alpha_{\ast}&>\frac{\alpha_{4}}{l}
\geq
\frac{\alpha_{4}}{l} \delta^{r}\left(- \ln(c \delta)\right)^{\kappa r},
\end{eqnarray*}
for $\delta$ small enough, say $\delta\leq t_{3}$.
In conclusion, the assertion follows for $\delta\leq \delta_{0}$, with $\delta_{0}:=\min\left\{ t_{2},t_{3}\right\}$ and $C:=\min\left\{\frac{C_{5}}{l}\left(\frac{k-1}{C_{4}}\right)^{r},\frac{\alpha_{4}}{l}\right\}$.
\end{proof}

The lower bound of Lemma~\ref{L:LowerEstAlphaAst} allows us to verify the next estimate in Lemma~\ref{L:AbschNorm1}, which is essential for the proof of the main result given in Theorem \ref{T:MainTheorem}.
\begin{lemma}\label{L:AbschNorm1}
Let Assumption \ref{Assumption} be satisfied and suppose that $\alpha_{\ast}$ is determined through the discrepancy principle in Definition~\ref{D:Discrepancy}. Then, we can conclude that
\[
\lVert u_{\alpha_{\ast}}^{\delta}-\overline{u}\rVert_{1} = \Or\left(\delta^{-\frac{1}{a}}\varphi\left(\delta\right)^{\frac{r}{2}}\right) \qquad \text{as }\delta \downarrow 0.
\]
\end{lemma}

\begin{proof}
Three cases have to be considered. In the first case, we consider $\alpha_{\ast}\in(0,\widetilde{\alpha}]$ with $\widetilde{\alpha}:=\min \left\{\alpha_{4},\lVert G\rVert \rme^{-ra\kappa }\right\}$, where $\alpha_{4}\leq \lVert G\rVert$ is the constant from Lemma~\ref{L:ZweiteAbsch}. The second case looks at $ \alpha_{\ast} \in (\widetilde{\alpha},\infty ) $, and the third case inspects the instance $\alpha_{\ast}=\infty$.
We start with the examination of the first case.
Rearranging the statement in Lemma~\ref{L:ZweiteAbsch}, we find out that
\[
\lVert u_{\alpha_{\ast}}^{\delta}-\overline{u}\rVert_{1} \leq C_{4}\alpha_{\ast}^{-\frac{1}{ra}}\varphi(\alpha_{\ast})+\frac{\delta}{\sqrt{\alpha_{\ast}}}, \qquad \alpha_{\ast}\leq\alpha_{4}.
\]
The term on the right-hand side is monotonically non-increasing for $\alpha_{\ast}\in \left(0,\lVert G\rVert \rme^{-ra\kappa}\right]$.
Thus, with the constraint $\alpha_{\ast}\leq \widetilde{\alpha}$, we derive by inserting the lower bound for $\alpha_{\ast}$ from Lemma~\ref{L:LowerEstAlphaAst} that
\begin{equation}\label{E:AbschNorm}
\lVert u_{\alpha_{\ast}}^{\delta}-\overline{u}\rVert_{1} \leq C_{4}\frac{\varphi\left( C\delta^{r}(-\ln(c\delta))^{\kappa r}\right)}{\left(C\delta^{r}(-\ln(c\delta))^{\kappa r}\right)^{\frac{1}{ra}}}+\frac{\delta }{ \sqrt{C\delta^{r}(-\ln(c\delta))^{\kappa r}}},
\end{equation}
with $\delta\leq\delta_{0}$ small enough. 
For the second summand in (\ref{E:AbschNorm}), we calculate
\begin{eqnarray*}
\frac{\delta }{ \sqrt{C\delta^{r}(-\ln(c\delta))^{\kappa r}}}
&=
C^{-\frac{1}{2}}\delta^{-\frac{1}{a}}\varphi(\delta)^{\frac{r}{2}}
=
\Or\left(\delta^{-\frac{1}{a}}\varphi\left(\delta\right)^{\frac{r}{2}}\right) \qquad \text{as }\delta \downarrow 0.
\end{eqnarray*}
For this reason, it is sufficient to verify the assertion for the fraction in the first summand of (\ref{E:AbschNorm}). 
The numerator yields 
\begin{align*}
 \varphi\left( C\delta^{r}(-\ln(c\delta))^{\kappa r}\right)
&=
\left[-\ln\left(cC\delta^{r}(-\ln(c\delta))^{\kappa r}\right)\right]^{-\kappa}\\
&=
\left[-\ln (c^{1-r}C)-r\ln(c\delta)-\kappa r\ln\left(-\ln(c\delta)\right) \right]^{-\kappa}\\
&=
\varphi(\delta)\left[\frac{\ln (c^{1-r}C)}{\ln(c\delta)}+r+\kappa r\frac{\ln\left(-\ln(c\delta)\right)}{\ln(c\delta)} \right]^{-\kappa} 
=
\varphi(\delta)g(\delta),
\end{align*}
with 
\[
g(\delta):=\left[\frac{\ln (c^{1-r}C)}{\ln(c\delta)}+r+\kappa r\frac{\ln\left(-\ln(c\delta)\right)}{\ln(c\delta)} \right]^{-\kappa}.
\]
For the denominator, we observe
\begin{eqnarray*}
\left(C\delta^{r}(-\ln(c\delta))^{\kappa r}\right)^{\frac{1}{ra}}
&=
C^{\frac{1}{ra}}\delta^{\frac{1}{a}}(-\ln(c\delta))^{\frac{\kappa}{a}}
=
C^{\frac{1}{ra}}\delta^{\frac{1}{a}}\varphi(\delta)^{-\frac{1}{a}}.
\end{eqnarray*}
Putting the components of the fraction together, we arrive at 
\[
C^{-\frac{1}{ra}}\delta^{-\frac{1}{a}}\varphi(\delta)^{\frac{1}{a}+1}g(\delta)=C^{-\frac{1}{ra}}\delta^{-\frac{1}{a}}\varphi(\delta)^{\frac{r}{2}}g(\delta),
\] 
where $\lim_{\delta\downarrow 0}g(\delta)=r^{-\kappa}$. Hence, we conclude 
\[
\lVert u_{\alpha_{\ast}}^{\delta}-\overline{u}\rVert_{1}
=\Or\left( \delta^{-\frac{1}{a}}\varphi(\delta)^{\frac{r}{2}}\right) \qquad \text{as }\delta\downarrow 0.
\]
Next, let $\alpha_{\ast}\in( \widetilde{\alpha},\infty)$.
Since 
\begin{eqnarray*}
\lVert F u_{\alpha_{\ast}}^{\delta}-f^{\delta}\rVert^{2}+\alpha_{\ast}\lVert u_{\alpha_{\ast}}^{\delta}-\overline{u}\rVert_{1}^{2} =T_{\alpha_{\ast}}^{\delta}(u_{\alpha_{\ast}}^{\delta})\leq T_{\alpha_{\ast}}^{\delta}(\overline{u})
=\lVert F \overline{u}-f^{\delta}\rVert^{2},
\end{eqnarray*}
we have
\[
\lVert u_{\alpha_{\ast}}^{\delta}-\overline{u}\rVert_{1}\leq \sqrt{\frac{1}{\alpha_{\ast}}}\lVert F \overline{u}-f^{\delta}\rVert=\Or(1) \qquad \text{as } \delta \downarrow 0.
\]
Thus especially
\[
 \lVert u_{\alpha_{\ast}}^{\delta}-\overline{u}\rVert_{1} =\Or\left( \delta^{-\frac{1}{a}}\varphi(\delta)^{\frac{r}{2}}\right)\qquad \text{as } \delta\downarrow 0,
\]
because 
$ \lim_{\delta\downarrow 0}\delta^{-\frac{1}{a}}\varphi(\delta)^{\frac{r}{2}}=\infty$. 

The third case $\alpha_{\ast}=\infty$ coincides with item $(a)$ of Definition \ref{D:Discrepancy}, respectively with the choice of $u_{\alpha_{\ast}}^{\delta}=\overline{u}$.
Thereupon, the claim of the theorem is an immediate consequence.
\end{proof}

\section{Main result}\label{S:MainResult}
With the previous results, we have everything at hand to prove our main result, which expresses the convergence rate of the discrepancy principle in the given setting. This convergence rate is asymptotically optimal according to \cite{Hohage2000}. Further, the same convergence rate is attained for an a priori parameter choice
\cite{HofmannPlato2020}.
\begin{theorem}[Main theorem]\label{T:MainTheorem}
Under Assumption \ref{Assumption}, the bound
\begin{eqnarray*}
\lVert u_{\alpha_{\ast}}^{\delta}-u^{\dagger}\rVert=
 \Or\left( \varphi(\delta)\right) \qquad \text{as } \delta\downarrow 0
\end{eqnarray*}
holds for $\alpha_{\ast}$ selected by the discrepancy principle specified in Definition \ref{D:Discrepancy}.
\end{theorem}
\begin{proof}
Using the triangle and interpolation inequality (\ref{E:InterpolationInequality}), we obtain
\begin{align*}
\lVert u_{\alpha_{\ast}}^{\delta}-u^{\dagger}\rVert
&\leq
\lVert u_{\alpha_{\ast}}^{\delta}-\hat{u}_{\beta}\rVert
+ \lVert \hat{u}_{\beta}-u^{\dagger}\rVert \\
&\leq
\underset{=:\textit{I}}{\underbrace{\lVert u_{\alpha_{\ast}}^{\delta}-\hat{u}_{\beta}\rVert_{-a}}}^{\frac{1}{a+1}}
\underset{=:\textit{II}}{\underbrace{\lVert u_{\alpha_{\ast}}^{\delta}-\hat{u}_{\beta}\rVert_{1}}}^{\frac{a}{a+1}}
+ \underset{=:\textit{III}}{\underbrace{\lVert \hat{u}_{\beta}-u^{\dagger}\rVert}},
\end{align*}
where the auxiliary element $\hat{u}_{\beta}$ is given as in Lemma~\ref{L:ABC}. 
First, we establish a bound for the expression in \textit{I}. The triangle inequality yields
\begin{eqnarray}\label{E:2}
\textit{I}=\lVert u_{\alpha_{\ast}}^{\delta}-\hat{u}_{\beta}\rVert_{-a}
\leq
\lVert u_{\alpha_{\ast}}^{\delta}-u^{\dagger}\rVert_{-a}
+
\lVert u^{\dagger}-\hat{u}_{\beta}\rVert_{-a}.
\end{eqnarray}
Now, the left-hand side of 
(\ref{E:nonlinearityCondition}) and the triangle inequality
are used for the first summand in (\ref{E:2}):
\[
\lVert u_{\alpha_{\ast}}^{\delta}-u^{\dagger}\rVert_{-a}
\leq \frac{1}{c_{a}}\lVert Fu_{\alpha_{\ast}}^{\delta}-f^{\dagger}\rVert
\leq \frac{1}{c_{a}}\left( \lVert Fu_{\alpha_{\ast}}^{\delta}-f^{\delta}\rVert+\delta \right).
\]
According to the discrepancy principle (\ref{E:DP}), we have $\lVert Fu_{\alpha_{\ast}}^{\delta}-f^{\delta}\rVert=\Or(\delta)$ as $\delta \downarrow 0$.
Thus
\[
\lVert u_{\alpha_{\ast}}^{\delta}-u^{\dagger}\rVert_{-a}=\Or(\delta) \qquad \text{as }\delta \downarrow 0.
\]
For the second summand in (\ref{E:2}),
item $\ref{En:itemiiLemmaABC}$ of Lemma~\ref{L:ABC} yields
\[
\lVert u^{\dagger}-\hat{u}_{\beta}\rVert_{-a}=\Or(\delta) \qquad \text{as } \delta \downarrow 0.
\]
It follows 
\[
\textit{I}=\Or(\delta) \qquad \text{as } \delta \downarrow 0.
\]
Next, we estimate the expression in \textit{II}. The triangle inequality yields
\begin{eqnarray}\label{E:3}
\textit{II}=\lVert u_{\alpha_{\ast}}^{\delta}-\hat{u}_{\beta}\rVert_{1} 
\leq
\lVert u_{\alpha_{\ast}}^{\delta}-\overline{u}\rVert_{1} 
+\lVert \overline{u}-\hat{u}_{\beta}\rVert_{1}.
\end{eqnarray}
For the first summand in (\ref{E:3}), it follows from Lemma~\ref{L:AbschNorm1} that
\[
\lVert u_{\alpha_{\ast}}^{\delta}-\overline{u}\rVert_{1} = \Or\left(\delta^{-\frac{1}{a}}\varphi\left(\delta\right)^{\frac{r}{2}}\right) \qquad \text{as }\delta \downarrow 0.
\]
For the second summand in (\ref{E:3}), item $\ref{En:itemiiiLemmaABC}$ of Lemma~\ref{L:ABC} gives
\[
\lVert \overline{u}-\hat{u}_{\beta}\rVert_{1}=\Or\left(\delta^{-\frac{1}{a}}\varphi(\delta)^{\frac{r}{2}}\right) \qquad \text{as } \delta \downarrow 0.
\]
Hence,
\[
\mathit{II}=\Or\left(\delta^{-\frac{1}{a}}\varphi(\delta)^{\frac{r}{2}}\right)=\Or\left(\delta^{-\frac{1}{a}}\varphi(\delta)^{\frac{a+1}{a}}\right) \qquad \text{as } \delta \downarrow 0.
\]
Finally, an estimate for the term in \textit{III} is given by item $\ref{En:itemiLemmaABC}$ of Lemma~\ref{L:ABC}:
\[
\mathit{III}=\lVert \hat{u}_{\beta}-u^{\dagger}\rVert=\Or\left(\varphi(\delta)\right) \qquad \text{as } \delta \downarrow 0.
\]
Putting the bounds together, the claim follows:
\begin{align*}
\lVert u_{\alpha_{\ast}}^{\delta}-u^{\dagger}\rVert
&\leq
\Or(\delta)^{\frac{1}{a+1}}\Or\left(\delta^{-\frac{1}{a}}\varphi(\delta)^{\frac{a+1}{a}}\right)^{\frac{a}{a+1}}+\Or\left(\varphi(\delta)\right) \\
&=
\Or\left(\varphi(\delta)\right) \ \text{ as } \delta \downarrow 0.
\end{align*}
\end{proof}

\section{Numerical results}\label{S:NumericalResults}
In this section, we present the numerical results realized by an algorithm that is practicable for the discrepancy principle given in Definition \ref{D:Discrepancy}.
The algorithm and the example are taken from \cite{HofmannPlato2020}.
For convenience, the example and the algorithm are reproduced below.

In the space $\ell_{2}=\left\{\left(u_{n}\right) \vert \ \lVert u\rVert_{\ell_{2}}^{2}=\sum_{n=1}^{\infty}u_{n}^{2} <\infty\right\}$, we consider the nonlinear operator \[
F:\ell_{2} \supset \mathcal{D}(F)\rightarrow \ell_{2}, \qquad (u_n)\mapsto 7\left(\frac{u_{n}}{n}\right)+\left(\frac{u_{n}^{2}}{n}\right),
\]
with domain
\[\mathcal{D}(F)=\left\{u \in \ell_{2} \vert \ \lVert u\rVert_{\ell_{2}}\leq 3 \right\}.\]
The operator 
\[
B:\ell_{2} \supset \mathcal{D}(B) \rightarrow \ell_{2}, \qquad (u_{n})\mapsto (nu_{n})
\]
with domain $\mathcal{D}(B)=\left\{(u_{n}) \vert \ (nu_{n})\in \ell_{2}\right\}$
generates the stronger norm $\lVert \cdot \rVert_{1}$.
We investigate the operator equation $Fu=f^{\dagger}$ with the solution
\[
u^{\dagger}=(u_{n}^{\dagger}), \qquad u_{1}^{\dagger}=1, \qquad u_{n}^{\dagger}=\frac{1}{\sqrt{n}(\ln (0.9^{-0.25} n))^{2.31}}, \qquad n\geq 2.
\]
It can be verified that this example satisfies Assumption \ref{Assumption}. Further, we have $u^{\dagger} \notin \mathcal{D}(B)$, and with $\overline{u}=0$, the source element
\[
 w=(w_{n}), \ \text{with } w_{1}=(-\ln(c))^{\kappa} \text{ and } w_{n}=\frac{4^{\kappa}}{\sqrt{n}(\ln (c^{-0.25}n))^{0.51}}, \qquad n\geq 2
\]
satisfies source condition (\ref{E:SourceCondition}) for $c=0.9$ and $\kappa=1.8$. 

The algorithm is implemented under the details stated below.
\begin{itemize}
\item To enable numerical computations, the space $\ell_{2}$ is discretized as $\mathbb{R}^{N}$ for $N=6000$.
\item The data $f^{\delta}$ is perturbed in the following way:
\[
f_{n}^{\delta}=f_{n}+\Delta_{n}, \qquad n=1,\dots,N,
\]
by uniformly distributed random values $\Delta_n$ satisfying $\vert \Delta_n\vert \leq \delta/\sqrt{N}$.
\item The initial constants in Algorithm \ref{Algo} are chosen as $\alpha^{(0)}=0.9$, $\theta=10$, and $k=3$.
\end{itemize}
\begin{algorithm}[H]\label{Algo}
\caption{Sequential discrepancy principle}
\LinesNumbered
\KwResult{Parameter $\alpha_{\ast}$ selected by the discrepancy principle in Definition \ref{D:Discrepancy}.}
Choose an initial guess $\alpha^{(0)}$ and constants $\theta,k>1$\;
\If{$\lVert Fu_{\alpha^{(0)}}^{\delta}-f^{\delta}\rVert \geq k\delta$}{\Repeat{$\lVert Fu_{\alpha^{(i)}}^{\delta}-f^{\delta}\rVert \leq k\delta \leq \lVert Fu_{\alpha^{(i-1)}}^{\delta}-f^{\delta}\rVert $ \mbox{$\mathrm{for \ the\ first\ time}$}}{$\alpha^{(i)}=\theta^{-i}\alpha^{(0)}$, $i=1,2,\dots$, \\ determine $u_{\alpha^{(i)}}^{\delta}$ and $\lVert Fu_{\alpha^{(i)}}^{\delta}-f^{\delta}\rVert$}}
set $\alpha_{\ast}=\alpha^{(i)}$;\\
\If{$\lVert Fu_{\alpha^{(0)}}^{\delta}-f^{\delta}\rVert \leq k\delta$}{\Repeat{$\lVert Fu_{\alpha^{(i-1)}}^{\delta}-f^{\delta}\rVert \leq k\delta \leq \lVert Fu_{\alpha^{(i)}}^{\delta}-f^{\delta}\rVert$  \mbox{$\mathrm{for \ the\ first\ time}$}}{$\alpha^{(i)}=\theta^{i}\alpha^{(0)}$, $i=1,2,\dots$, \\ determine $u_{\alpha^{(i)}}^{\delta}$ and $\lVert Fu_{\alpha^{(i)}}^{\delta}-f^{\delta}\rVert$}}
set $\alpha_{\ast}=\alpha^{(i-1)}$;\\
\Return{$\alpha_{\ast}$} 
\end{algorithm}
\vspace{0.3cm}

\begin{table}[H]
\caption{Numerical results of Algorithm \ref{Algo}}\label{Table}
\begin{center}
\begin{tabular}{@{}ccccc}
\toprule
$\mathbf{\delta}$  & $\alpha_{\ast}$ &$\mathbf{\lVert u_{\alpha_{\ast}}^{\delta}-u^{\dagger}\rVert_{\ell_{2}}}$ &$\varphi(\delta)$ & $\frac{\mathbf{\lVert u_{\alpha_{\ast}}^{\delta}-u^{\dagger}\rVert_{\ell_{2}}}}{\varphi(\delta)}$ \\
\midrule
$1.000\cdot 10^{-3} $& $9\cdot 10^{-6}$ &	$0.039\thinspace932$	& 0.0300 & 1.3303 \\
$5.000\cdot 10^{-4}$&$9\cdot 10^{-6}$&		$ 0.039\thinspace935 $ & 0.0253 &1.5764\\	
$ 2.500\cdot 10^{-4}$&$9\cdot 10^{-7}$& $0.030\thinspace659 $&  0.0217  &   1.4132  \\
$1.250\cdot 10^{-4}$&$9\cdot 10^{-8}$ & $ 0.024\thinspace008$ & 0.0188 &  1.2764 \\
$6.250 \cdot 10^{-5}$& $9\cdot 10^{-9}$ & $ 0.018\thinspace967 $ & $ 0.0165$ &   1.1509\\
$3.125\cdot 10^{-5}$&  $9\cdot 10^{-10}$ & $ 0.014\thinspace946$ &$ 0.0146$ &  1.0259 \\
$ 1.563\cdot 10^{-5}$&$9\cdot 10^{-11}$ & $ 0.011\thinspace578$&$ 0.0130$ &  0.8918 \\
$7.813\cdot 10^{-6}$&$9\cdot 10^{-12}$& $ 0.008\thinspace572$ &$ 0.0116 $ & 0.7358\\
$ 3.906\cdot 10^{-6}$&$9\cdot 10^{-12}$& $ 0.008\thinspace572$& $0.0105$ &  0.8150\\
$ 1.953\cdot 10^{-6}$&$9\cdot 10^{-13}$& $ 0.005\thinspace622$ &$0.0095$ &  0.5888\\
\bottomrule
\end{tabular}
\end{center}
\end{table}

Table \ref{Table} presents the results for Algorithm \ref{Algo}. The second column illustrates the calculated values of $\alpha_{\ast}$ for the different noise levels in Column 1. The third and the fifth column confirm the assertion of Theorem \ref{T:MainTheorem}.

\section*{Acknowledgement}
This paper was created as part of the authors’ DFG-Project No.~453804957, supported
by the German Research Foundation under grant PL 182/8-1.

\end{document}